\documentclass[11pt]{article}
\usepackage[T1]{fontenc}
\usepackage{lmodern}
\usepackage{amsmath,amssymb,amsthm}
\usepackage{graphicx}
\usepackage[margin=1in]{geometry}
\usepackage{microtype}
\usepackage[hidelinks]{hyperref}

\newtheorem{theorem}{Theorem}[section]
\newtheorem{lemma}[theorem]{Lemma}
\newtheorem{corollary}[theorem]{Corollary}
\theoremstyle{definition}
\newtheorem{problem}{Problem}
\DeclareMathOperator{\supp}{supp}
\DeclareMathOperator{\dist}{dist}
\DeclareMathOperator{\Int}{int}
\newcommand{\R}{\mathbb R}
\newcommand{\Z}{\mathbb Z}
\newcommand{\T}{\mathbb T}
\newcommand{\norm}[1]{\lVert #1\rVert}

\title{Regular closed gradient ranges and asymptotic gradient values}
\author{Tao Hu}
\date{}
\hypersetup{
  pdftitle={Regular closed gradient ranges and asymptotic gradient values},
  pdfauthor={Tao Hu},
  pdfkeywords={Gradient ranges, regular closed sets, asymptotic gradient values,
    Lagrangian submanifolds, symplectic homeomorphisms}
}

\begin{document}
\maketitle

\begin{abstract}
We prove that the gradient range of every nonzero compactly supported
$C^1$ function on $\R^n$ is the closure of its interior, for every $n\ge1$.
More generally, for an arbitrary $C^1$ function, the closure of its
gradient range is the union of the closure of the range's interior and
the set of asymptotic gradient values. We also prove that compact smooth
Lagrangian submanifolds without boundary in $\R^{2n}$ and their images
under symplectic homeomorphisms have regular closed momentum projections.
\end{abstract}

{\small
\noindent\textbf{Mathematics Subject Classification (2020).}
Primary 26B05; Secondary 53D12, 53D05.\par
\smallskip
\noindent\textbf{Keywords.}
Gradient ranges; regular closed sets; asymptotic gradient values;
\mbox{Lagrangian submanifolds}; symplectic homeomorphisms.\par
}

\section{Introduction}

For a real-valued function $f$ on $\R^n$, its support is
$\supp f=\overline{\{x:f(x)\ne0\}}$. We call $f$ a \emph{bump}
if it is nonzero and has compact support. For $f\in C^1(\R^n)$, write
\[
G_f=\{\nabla f(x):x\in\R^n\}.
\]
Which subsets of $\R^n$ can occur as $G_f$ when $f$ is a continuously
differentiable bump? The range is compact and connected, and the origin
belongs to its interior. It need not be convex. Early constructions of
bumps with prescribed gradient ranges arose in the study of generalized
derivatives; see Borwein, Borwein, and Wang \cite[Section~3.1]{BBW1996}.
Borwein, Fabian, Kortezov, and Loewen \cite{BFKL2001} subsequently
proved local restrictions on such ranges and constructed bumps realizing,
in particular, finite unions of strictly convex compact sets with
nonempty interiors, provided that the interior of the union is connected
and contains the origin.

A closed set is called \emph{regular closed} if it equals the closure of
its interior. The following question was raised in
\cite{BFKL2001} and restated by Hiriart-Urruty
\cite[Problem~4]{HiriartUrruty2007}.
It is also stated in Kol\'a\v{r} and Kristensen's preprint
\cite{KolarKristensen2002} and in Rifford \cite{Rifford2003}.

\begin{problem}\label{prob:gradient-range}
Let $n\ge1$ and let $f\in C^1(\R^n)$ be a bump. Is it true that
\[
G_f=\overline{\Int G_f}?
\]
\end{problem}

In dimension one the answer follows from the intermediate value
theorem. Gaspari \cite[Theorem~2.1]{Gaspari2002} proved the assertion for $C^2$ bumps
on the plane. Rifford \cite[Theorem~0.1]{Rifford2003} proved it for
$C^{n+1}$ bumps in arbitrary finite dimension, using a geometric argument
and a Sard-type refinement that he quotes from Federer. For the classical
critical-value theorem, see \cite{Sard1942}. The $C^2$ case in every
finite dimension follows from a theorem of Hartman and Nirenberg
\cite[Section~3, Theorem~I]{HartmanNirenberg1959}
and the inverse function theorem. In the plane, Kol\'a\v{r} and
Kristensen \cite{KolarKristensen2005} proved the conclusion when
$|\nabla f(x)-\nabla f(y)|\le\omega(|x-y|)$ for a modulus of continuity
$\omega$ satisfying $\omega(t)/\sqrt t\to0$ as $t\downarrow0$.
For planar $C^1$ bumps without an additional regularity assumption,
see Korobkov \cite[Corollary~1.3]{Korobkov2007}.

The following theorem answers Problem~\ref{prob:gradient-range}
affirmatively in every finite dimension, including $n\ge3$ under
the $C^1$ hypothesis alone.

\begin{theorem}\label{thm:main}
Let $n\ge1$, and let $f\in C^1(\R^n)$ be a bump. Then
\[
G_f=\overline{\Int G_f}.
\]
\end{theorem}

The argument gives a more general statement without any assumption on
support or growth. For $f\in C^1(\R^n)$, define
\begin{equation}\label{eq:asymptotic}
A_\infty(f)=\bigcap_{R>0}
\overline{\{\nabla f(x):|x|\ge R\}}.
\end{equation}
Thus $p\in A_\infty(f)$ precisely when there is a sequence $(x_k)$
with $|x_k|\to\infty$ and $\nabla f(x_k)\to p$. These are the
\emph{asymptotic values} of the gradient map. The set $A_\infty(f)$
is closed and may be empty.

\begin{theorem}\label{thm:asymptotic}
Let $n\ge1$ and let $f\in C^1(\R^n)$. Then
\[
\overline{G_f}=\overline{\Int G_f}\cup A_\infty(f).
\]
\end{theorem}

In particular, if $p\in G_f$ and $(\nabla f)^{-1}(B(p,r))$ is bounded
for some $r>0$, then $p\in\overline{\Int G_f}$.

For $C^2$ functions this assertion also follows from
\cite[Section~3, Theorem~I]{HartmanNirenberg1959}.
Indeed, suppose that $p\in G_f\setminus\overline{\Int G_f}$ and that the inverse
image of a small open ball $B$ about $p$ is bounded. Choose $B$
disjoint from $\Int G_f$, and let $D$ be the component of this inverse
image containing a preimage of $p$.
Continuity gives $\nabla f(\partial D)\subset\partial B$.
The inverse function theorem forces
$\det D^2f=0$ on $D$, and Hartman--Nirenberg gives
$\nabla f(D)\subset\nabla f(\partial D)$, a contradiction.

In terms of optimization, $G_f$ is the set of parameters $p$ for which
the linearly perturbed function $x\mapsto f(x)-p\cdot x$ has a stationary
point. Theorem~\ref{thm:asymptotic} shows that every parameter in
$G_f\setminus A_\infty(f)$ is a limit of parameters around which
stationary points exist for all sufficiently small parameter perturbations.

For a bump, $A_\infty(f)=\{0\}$, and the
origin already belongs to $\Int G_f$. A further application is
the following extension of Theorem~\ref{thm:main}.

\begin{corollary}\label{cor:infinity}
Let $f\in C^1(\R^n)$ be non-affine, and suppose that
$\nabla f(x)\to p_\infty\in\R^n$ as $|x|\to\infty$. Then
\[
\overline{G_f}=\overline{\Int G_f}.
\]
\end{corollary}

The closure on the left is necessary: for $f(x)=\arctan x$ on $\R$,
the derivative range is $(0,1]$. The non-affine assumption excludes a
singleton gradient range.

Our proof combines an elementary construction with a rigidity theorem
from symplectic topology. A compact set with empty interior admits
smooth functions of arbitrarily small amplitude whose gradients are
arbitrarily large on that set. If a portion of a gradient range had
empty interior and all its preimages stayed bounded, these functions
would generate Hamiltonian displacements of arbitrarily small energy
from the gradient graph. Usher's displacement rigidity results
\cite{Usher2014,Usher2022} exclude this possibility. A cotangent
embedding and invariance under symplectic homeomorphisms extend the
required rigidity to the graph of an arbitrary $C^1$ function.

The same argument applies to projections of compact Lagrangian
submanifolds. On $\R^{2n}=\R_x^n\times\R_p^n$ use the symplectic form
$\omega_0=\sum_{j=1}^n dx_j\wedge dp_j$ and the projection
$\pi_p(x,p)=p$. A smooth $n$-dimensional submanifold is
\emph{Lagrangian} if $\omega_0$ vanishes on each of its tangent spaces.
A \emph{symplectic homeomorphism} is a homeomorphism that is a locally
uniform limit of smooth diffeomorphisms preserving $\omega_0$.

\begin{theorem}\label{thm:lagrangian}
Let $L_0\subset\R^{2n}$ be a compact smooth Lagrangian submanifold
without boundary, and let $\psi:\R^{2n}\to\R^{2n}$ be a symplectic
homeomorphism. Then
\[
\pi_p(\psi(L_0))=\overline{\Int\bigl(\pi_p(\psi(L_0))\bigr)}.
\]
\end{theorem}

Section~\ref{sec:projections} proves the auxiliary construction and a
projection lemma, and deduces Theorem~\ref{thm:lagrangian}.
Section~\ref{sec:gradients} establishes the required rigidity of
gradient graphs and proves Theorems~\ref{thm:asymptotic}
and~\ref{thm:main} and Corollary~\ref{cor:infinity}.

We use $|\cdot|$ for the Euclidean norm, $B(p,r)$ for the open ball
with centre $p$ and radius $r$, and $B_r=B(0,r)$. The notation
$\norm{\cdot}_\infty$ denotes the uniform norm on the indicated domain.
The interior and closure of subsets of $\R^n$ are taken in the Euclidean
topology.

\section{Displacement energy and projections}\label{sec:projections}

The first lemma is obtained by smoothing a sum of triangular waves.
We translate the waves so that, at every point of the prescribed set,
at least one summand is differentiable.

\begin{lemma}\label{lem:oscillation}
Let $K\subset\R^n$ be compact with empty interior. For every
$\varepsilon,D>0$, there exists a smooth function $H:\R^n\to\R$,
periodic in each coordinate, such that
\[
\norm{H}_\infty<\varepsilon,
\qquad |\nabla H(p)|>D\quad(p\in K).
\]
\end{lemma}

\begin{proof}
The assertion is immediate if $K=\varnothing$, so assume that $K$ is
nonempty. Choose $\lambda>D$ and $\delta>0$ such that
$n\lambda\delta<\varepsilon$.

There exists $a\in(0,\delta)^n$ for which
\[
K\cap(a+\delta\Z^n)=\varnothing.
\]
Indeed, the forbidden translations in this cube belong to the sets
$K-\delta m$, $m\in\Z^n$. Only finitely many of these sets meet the cube,
because $K$ is bounded. Each is closed with empty interior, so their finite
union does not cover the cube. Since the translated lattice is closed,
compactness gives
\[
d=\dist(K,a+\delta\Z^n)>0.
\]

Consider the periodic Lipschitz function
\[
h(p)=\lambda\sum_{j=1}^n\dist(p_j,a_j+2\delta\Z).
\]
It satisfies $0\le h\le n\lambda\delta$. Each summand is a triangular wave,
whose derivative is $\lambda$ or $-\lambda$ except at its breakpoints
$a_j+\delta\Z$.

Let $H=h*\varphi$, where $\varphi$ is a nonnegative smooth mollifier of
integral one supported in $B_\tau$, with $\tau<d/(2\sqrt n)$.
Fix $p\in K$. The distance to the product lattice satisfies
\[
\sum_{j=1}^n\dist(p_j,a_j+\delta\Z)^2
=\dist(p,a+\delta\Z^n)^2\ge d^2.
\]
Hence some coordinate $j$ has distance at least $d/\sqrt n$ from its
breakpoints. Throughout $B(p,\tau)$, the
$j$th partial derivative of $h$ is therefore one fixed number, either
$\lambda$ or $-\lambda$. Consequently
$|\partial_jH(p)|=\lambda>D$. Finally,
$0\le H\le n\lambda\delta<\varepsilon$, and $H$ is smooth and
$2\delta$-periodic in each coordinate. All its derivatives are bounded.
\end{proof}

We next recall the displacement energy used in the argument. The two
ambient spaces we need are $\R^{2n}$ and
$\T^n\times\R^n$, where $\T^n=\R^n/\Z^n$. The latter is the
cotangent bundle of the flat torus, with coordinates $(\theta,\eta)$
and symplectic form $\sum_j d\theta_j\wedge d\eta_j$.
Lagrangian submanifolds in this space are defined using this form.
For either space $M$, a smooth Hamiltonian
$F:[0,1]\times M\to\R$ with compact support generates a flow
$\phi_F^t$. In Euclidean coordinates our convention is
\[
\dot x=\nabla_pF(t,x,p),\qquad
\dot p=-\nabla_xF(t,x,p).
\]
Its Hofer length is
\[
\mathcal L(F)=\int_0^1
\left(\max_M F(t,\cdot)-\min_M F(t,\cdot)\right)\,dt.
\]
For an open set $U\subset M$ and a closed set $N\subset M$, define
\[
e_M(U,N)=\inf\{\mathcal L(F):\phi_F^1(\overline U)\cap N=\varnothing\},
\]
where the infimum is over these Hamiltonians and is $+\infty$ if the
class is empty. We omit the subscript when $M=\R^{2n}$.

Taking the infimum over Hamiltonians generating each time-one map
recovers the equivalent definition in terms of Hofer norms used by Usher.

We use two known facts. Both ambient spaces are geometrically bounded:
their standard compatible complex structures and complete flat product
metrics have uniform compatibility constants, zero curvature, and
positive injectivity radius. First, for either of these ambient spaces,
the result of Usher \cite[Lemma~2.4]{Usher2022}, based on
\cite[Corollary~4.10]{Usher2014}, applies to every compact smooth
Lagrangian submanifold $L$ without boundary. For every open set $U$
meeting $L$, it gives\footnote{In the auxiliary product-torus construction
in the proof of \cite[Corollary~4.10]{Usher2014}, the radius of each circle
can be taken to be $r/(2\sqrt n)$ in place of $r/2$. The product torus then
lies at Euclidean distance $r/2$ from the origin, inside the required
$2n$-dimensional ball of radius $r$, and retains the stated Lagrangian
and transverse-intersection properties.}
\begin{equation}\label{eq:rigidity}
e_M(U,L)>0.
\end{equation}
Second, \cite[Proposition~1.4]{Usher2022} shows that the property
$e(U,N)>0$ for every open set $U$ meeting a closed set $N$ is preserved
by symplectic homeomorphisms.

The next lemma requires compactness only over a neighborhood of the
specified projection value.

\begin{lemma}\label{lem:projection}
Let $N\subset\R^{2n}$ be closed, and suppose that $e(U,N)>0$ for
every bounded open set $U$ meeting $N$. Let $p_0\in\pi_p(N)$.
If $N\cap\pi_p^{-1}(\overline{B(p_0,r_0)})$ is compact for some
$r_0>0$, then
\[
p_0\in\overline{\Int\bigl(\pi_p(N)\bigr)}.
\]
\end{lemma}

\begin{proof}
Set $E=\pi_p(N)$. Suppose that $p_0\notin\overline{\Int E}$.
Choose $0<r<r_0$ such that $B(p_0,2r)\cap\Int E=\varnothing$.
The set
\[
K=E\cap\overline{B(p_0,r)}
\]
is compact, since it is the projection of the compact set
$N\cap\pi_p^{-1}(\overline{B(p_0,r)})$. It has empty interior.
Choose $(x_0,p_0)\in N$, and choose $R>0$ such that
\[
N\cap\pi_p^{-1}(\overline{B(p_0,r)})\subset B_R\times\R^n.
\]
Fix an open neighborhood $U$ of $(x_0,p_0)$ with
$\overline U\subset B_R\times B(p_0,r)$.

Given $\varepsilon>0$, Lemma~\ref{lem:oscillation} gives a smooth
function $H$, periodic in each coordinate, such that
\[
\norm{H}_\infty<\varepsilon,
\qquad |\nabla H(p)|>2R\quad(p\in K).
\]
Put $A=\norm{\nabla H}_\infty$. Choose compactly supported smooth
functions $\alpha,\beta:\R^n\to[0,1]$ that equal one on neighborhoods
of $\overline{B_{R+A}}$ and $\overline{B(p_0,r)}$, respectively, and set
\[
F(x,p)=\alpha(x)\beta(p)H(p).
\]
This autonomous Hamiltonian has $\mathcal L(F)<2\varepsilon$.
For $(x,p)\in\overline U$, the entire path
$(x+t\nabla H(p),p)$, $0\le t\le1$, lies where both cutoffs equal one.
It therefore solves Hamilton's equations for $F$, and
\begin{equation}\label{eq:shear}
\phi_F^1(x,p)=(x+\nabla H(p),p)
\qquad((x,p)\in\overline U).
\end{equation}
If the endpoint in \eqref{eq:shear} belonged to $N$, then $p\in K$
and its first coordinate would lie in $B_R$. Since $x\in B_R$ as
well, this would give $|\nabla H(p)|<2R$, a contradiction. Hence
$\phi_F^1(\overline U)\cap N=\varnothing$ and
$e(U,N)<2\varepsilon$. The neighborhood $U$ is independent of
$\varepsilon$, so this contradicts $e(U,N)>0$.
\end{proof}

We can now prove Theorem~\ref{thm:lagrangian}.

\begin{proof}
Set $L=\psi(L_0)$. By \eqref{eq:rigidity} and its invariance under
symplectic homeomorphisms, $e(U,L)>0$ whenever $U$ meets $L$.
Compactness of $L$ allows us to apply Lemma~\ref{lem:projection} at
every point of $\pi_p(L)$. This gives
$\pi_p(L)\subset\overline{\Int\pi_p(L)}$.
The reverse inclusion follows because $\pi_p(L)$ is compact.
\end{proof}

\section{Gradient graphs}\label{sec:gradients}

To apply Lemma~\ref{lem:projection} to a gradient range, we need
every point of the gradient graph to be rigid with respect to that
graph, in the sense of \cite[Definition~1.1]{Usher2022}.
Local rigidity would not suffice, since the Hamiltonians in that lemma
need not be supported in a fixed small neighborhood. We obtain the
required statement by embedding $\R^{2n}$ into the cotangent bundle
of a torus.

\begin{lemma}\label{lem:graph}
Let $f\in C^1(\R^n)$, and set
\[
\Gamma_f=\{(x,\nabla f(x)):x\in\R^n\}.
\]
For every bounded open set $U\subset\R^{2n}$ meeting $\Gamma_f$,
we have $e(U,\Gamma_f)>0$.
\end{lemma}

\begin{proof}
First consider the zero section $Z=\R^n\times\{0\}$.
Choose a smooth increasing diffeomorphism
$a:\R\to(-1/4,1/4)$ with $a'>0$, and define
\[
\iota:\R^{2n}\longrightarrow M=\T^n\times\R^n,
\qquad
\iota(x,p)=\left((a(x_j))_{j=1}^n,
\left(\frac{p_j}{a'(x_j)}\right)_{j=1}^n\right).
\]
The first coordinates are read modulo $\Z^n$.
A direct calculation gives
$\iota^*(\sum_j d\theta_j\wedge d\eta_j)=\omega_0$.
Thus $\iota$ is a symplectic diffeomorphism onto the open set
$V=(-1/4,1/4)^n\times\R^n\subset M$, and
$\iota(Z)=V\cap L$, where $L=\T^n\times\{0\}$ is a compact
smooth Lagrangian submanifold.

The quotient map is injective on $[-1/4,1/4]^n$, since distinct points
of this cube cannot differ by an integer vector.
Figure~\ref{fig:torus-chart} shows this coordinate patch for $n=2$.

\begin{figure}[htbp]
\centering
\includegraphics[width=\textwidth]{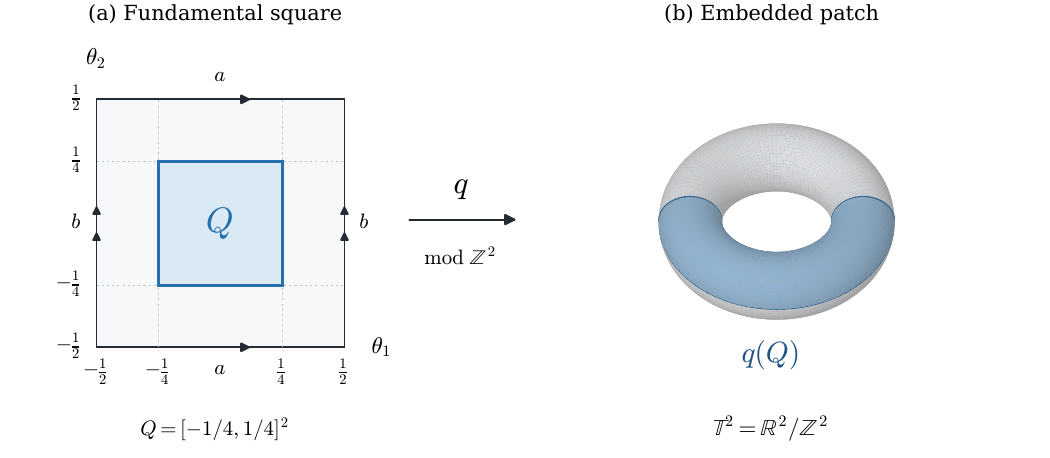}
\caption{The quotient map $q:\R^2\to\T^2=\R^2/\Z^2$ identifies
opposite sides of the outer square $[-1/2,1/2]^2$ with the indicated
orientations. It embeds the blue square $Q=[-1/4,1/4]^2$ into the
torus; no boundary points of $Q$ are identified. The proof uses the
open coordinate patch $q(\Int Q)$. Only the base torus is shown;
cotangent fibers are omitted.}
\label{fig:torus-chart}
\end{figure}

Let $U\subset\R^{2n}$ be bounded and open with $U\cap Z\ne\varnothing$.
Writing the superscript $M$ for closure in the ambient manifold,
compactness of $\overline U$ gives
\[
\overline{\iota(U)}^{\,M}=\iota(\overline U).
\]

Let $F:[0,1]\times\R^{2n}\to\R$ be any smooth compactly supported
Hamiltonian satisfying $\phi_F^1(\overline U)\cap Z=\varnothing$.
Choose a compact set $K_F\subset\R^{2n}$ containing the support of
$F(t,\cdot)$ for every $t\in[0,1]$, and define
\[
\widehat F(t,y)=
\begin{cases}
F(t,\iota^{-1}(y)),&y\in V,\\
0,&y\in M\setminus V.
\end{cases}
\]
Since $\iota(K_F)$ is a compact subset of $V$, extension by zero gives
a smooth compactly supported Hamiltonian on $M$. Its maximum and
minimum at each time agree with those of $F$: the extension adds
only the value zero, already attained outside $K_F$. Hence
$\mathcal L(\widehat F)=\mathcal L(F)$.

The flow of $\widehat F$ fixes $M\setminus V$ pointwise and therefore
preserves $V$. Symplecticity of $\iota$ and uniqueness of the flow give
\[
\phi_{\widehat F}^t\circ\iota
=\iota\circ\phi_F^t,\qquad 0\le t\le1.
\]
The displaced set remains in $V$, where $L\cap V=\iota(Z)$.
Consequently,
\[
\phi_{\widehat F}^1
  \bigl(\overline{\iota(U)}^{\,M}\bigr)
=\iota\bigl(\phi_F^1(\overline U)\bigr)\subset V\setminus L.
\]

Taking the infimum over all such $F$ gives
\[
e(U,Z)\ge e_M(\iota(U),L)>0.
\]
The strict inequality follows from \eqref{eq:rigidity}, since
$\iota(U)$ is open in $M$ and meets the compact Lagrangian $L$.
Finally, every open set $W$ meeting $Z$ contains a bounded open set
$U$ meeting $Z$. A displacement of $\overline W$ from $Z$ also
displaces $\overline U$, and hence $e(W,Z)\ge e(U,Z)>0$.

Now consider the homeomorphism
\[
\psi_f(x,p)=(x,p+\nabla f(x)),
\qquad
\psi_f^{-1}(x,p)=(x,p-\nabla f(x)).
\]
Smooth mollifications $f_\nu$ converge to $f$ in $C^1$ on compact sets.
The maps $\psi_{f_\nu}$ are symplectic diffeomorphisms by symmetry of
the Hessians and converge locally uniformly to $\psi_f$.
Thus $\psi_f$ is a symplectic homeomorphism. The graph $\Gamma_f$ is
closed by continuity of $\nabla f$. Since $\Gamma_f=\psi_f(Z)$,
Usher's invariance result \cite[Proposition~1.4]{Usher2022} gives
$e(U,\Gamma_f)>0$ for every bounded open set $U$ meeting $\Gamma_f$.
\end{proof}

We first prove the assertion about asymptotic values in
Theorem~\ref{thm:asymptotic}.

\begin{proof}
Let $p_0\in G_f\setminus A_\infty(f)$. By \eqref{eq:asymptotic},
there are $r,R>0$ such that
\[
|\nabla f(x)-p_0|>r\qquad(|x|\ge R).
\]
Hence the closed set
$\Gamma_f\cap\pi_p^{-1}(\overline{B(p_0,r)})$ is bounded and
therefore compact. By Lemma~\ref{lem:graph}, $e(U,\Gamma_f)>0$ for
every bounded open set $U\subset\R^{2n}$ meeting $\Gamma_f$.
We may thus apply Lemma~\ref{lem:projection} with $N=\Gamma_f$.
Since $\pi_p(\Gamma_f)=G_f$, we obtain
$p_0\in\overline{\Int G_f}$.
Thus $G_f\subset\overline{\Int G_f}\cup A_\infty(f)$.
The right-hand side is closed and is contained in $\overline{G_f}$,
so taking closures proves the stated equality.
\end{proof}

For completeness, we give the elementary argument at the origin that
completes the proof of Theorem~\ref{thm:main}.

\begin{proof}
For a bump $f$, the set $G_f$ is compact and $A_\infty(f)=\{0\}$.
It remains, by Theorem~\ref{thm:asymptotic}, to show that
$0\in\Int G_f$. After replacing $f$ by $-f$ if necessary, choose
$x_*$ with $f(x_*)>0$, and let $B_R$ contain $\supp f$ in its
interior. For every sufficiently small $q\in\R^n$, the value of
$f(x)-q\cdot x$ at $x_*$ exceeds all its values on $\partial B_R$.
Its maximum on $\overline{B_R}$ is therefore attained in the interior,
where $\nabla f=q$. This proves the required inclusion of a ball
about zero in $G_f$.
\end{proof}

\begin{samepage}
Finally, we prove Corollary~\ref{cor:infinity}.

\begin{proof}
Here $A_\infty(f)=\{p_\infty\}$. By Theorem~\ref{thm:asymptotic},
it suffices to show that $p_\infty\in\overline{\Int G_f}$.
Because $f$ is non-affine, the continuous function
$x\mapsto|\nabla f(x)-p_\infty|$ takes some value $a>0$ and tends
to zero at infinity. The intermediate value theorem therefore gives
a sequence $(y_k)$ with
\[
y_k\in G_f\setminus\{p_\infty\},
\qquad |y_k-p_\infty|=\frac{a}{k+1}\longrightarrow0.
\]
By Theorem~\ref{thm:asymptotic}, every $y_k$ belongs to
$\overline{\Int G_f}$. Since this set is closed and $y_k\to p_\infty$,
it also contains $p_\infty$, as required.
\end{proof}
\end{samepage}

\bigskip
\begin{samepage}
\footnotesize
\noindent\textsc{Department of Information, Risk, and Operations Management,
McCombs School of Business,\\
The University of Texas at Austin,
Austin, TX 78712, USA}\par
\nopagebreak
\noindent\textit{Email address:}
\href{mailto:tao_hu@utexas.edu}{\texttt{tao\_hu@utexas.edu}}\par
\end{samepage}

\end{document}